\documentclass[10pt,leqno,draft]{amsart}

\usepackage {amsfonts}
\usepackage{amsthm}
\usepackage{amsmath}
\usepackage{comment}

\theoremstyle{plain}
\newtheorem{tw}{Theorem}[section]

\newtheorem {lem} [tw]{Lemma}
\newtheorem {propn}[tw] {Proposition}
\newtheorem{cor}[tw]{Corollary}

\theoremstyle{definition}
\newtheorem {deft}[tw] {Definition}

\newcommand{\bc} {\mathbb{C}}

\newcommand{\br}{\mathbb{R}}

\newcommand{\alg} {\mathsf{A}}

\newcommand {\id} {{\mathrm{id}}}

\newcommand{\clg} {\mathsf{C}}

\newcommand{\dlg} {\mathsf{D}}

\newcommand{\blg}{\mathsf{B}}

\newcommand{\Alg}{\mathcal{A}}

\newcommand{\tu}{\textup}

\newcommand{\Mpsi}{\mathfrak{M}_{\psi}}
\newcommand{\Npsi}{\mathfrak{N}_{\psi}}

\newcommand{\Cou}{\epsilon}
\newcommand{\Com}{\Delta}

\newenvironment{rlist}
{

\begin{enumerate}}
{\end{enumerate}}

\newcommand{\ot}{\otimes}

\numberwithin{equation}{section}

\let\cop\Com
\DeclareMathOperator\dom{dom}
\DeclareMathOperator\clsp{\overline{\text{span}}}

\newcommand\set[2]{\{\,#1: #2\,\}}
\newcommand{\pair}[1]{\langle #1\rangle}
\newcommand{\Mphi}{\mathfrak{M}_{\phi}}
\newcommand{\Nphi}{\mathfrak{N}_{\phi}}
\newcommand{\sub}{\subset}
\newcommand{\bus}{\supset}
\newcommand{\inv}{^{-1}}
\newcommand{\conj}{\overline}

\newcommand\complex{\mathbb{C}}

\begin{document}

\author{Pekka Salmi}
\address{Department of Pure Mathematics,
University of Waterloo, Waterloo, ON, N2L 3G1, Canada}
\email{pekka.salmi@iki.fi}

\author{Adam Skalski}
\address{Institute of Mathematics of the Polish Academy of Sciences,
ul.\'Sniadeckich 8, 00-956 Warszawa, Poland}
\email{a.skalski@impan.pl}

\title{\bf Idempotent states on locally compact quantum groups}


\begin{abstract}
\noindent Idempotent states on a unimodular coamenable locally compact
quantum group $\alg$ are shown to be in one-to-one
correspondence with right invariant expected $C^*$-subalgebras of
$\alg$. Haar idempotents, that is, idempotent
states arising as Haar states on compact quantum subgroups of $\alg$,
are characterised and shown to be invariant
under the natural action of the modular element. This leads to the
one-to-one correspondence between
Haar idempotents on $\alg$ and  right invariant symmetric expected
$C^*$-subalgebras of $\alg$ without the
unimodularity assumption. Finally the tools developed in the first part of the paper are applied to show that the coproduct of a coamenable locally compact quantum group restricts to a continuous coaction on each right invariant expected  $C^*$-subalgebra.
\end{abstract}

\maketitle

Idempotent probability measures on locally compact groups arise naturally as limit distributions of random walks.
By analogy, when one considers quantum random walks in the setup
provided by topological quantum groups
\cite{franz+skalski08},
one is led to consider idempotent states on locally compact quantum
groups. Since the work of Kawada and It\^o \cite{KawadaIto},
idempotent probability measures have been well understood, as they
all arise as Haar measures on compact subgroups. In the quantum world,
as shown in \cite{Pal}, the situation is
more complicated, as already some finite quantum groups admit
idempotent states which cannot be canonically
associated with any quantum subgroup. Motivated by this discovery
U.\,Franz and the second-named author, later
joined by R.\,Tomatsu, have begun a systematic investigation of
idempotent states on finite and compact quantum
groups [FS$_{2-3}$, FST]. In particular, necessary and
sufficient conditions for such states to be
\emph{Haar idempotents}, i.e.\ to arise as Haar states on closed quantum subgroups, have been identified, and
close relations to expected right invariant $C^*$-subalgebras (called in \cite{FS} coidalgebras) uncovered. On the
other hand the first-named author, inspired by the harmonic analysis considerations due to Lau and Losert  showed
in a recent paper \cite{S}
a one-to-one correspondence between compact quantum subgroups of a coamenable locally
compact quantum group $\alg$ and certain right invariant unital $C^*$-subalgebras of $\alg$.

Motivated by these developments, in this paper we study idempotent
states and related structures on coamenable
\emph{locally compact} quantum groups in the sense of \cite{KV}. At
first it might appear that the algebraic
formalism here is similar to that encountered in [FS$_{2-3}$] and
\cite{FST}, but technical aspects of the locally
compact theory (such as the absence of a natural dense Hopf
$^*$-algebra) make the problems we investigate more
complicated and necessitate developing new approaches. In particular
we discover an unexpected role played by
unimodularity: our proof of the main result characterising idempotent
states in terms of right invariant expected
$C^*$-subalgebras requires the assumption that the left and right Haar
weights coincide. For Haar idempotents this
restriction is not necessary, as we show that they are invariant in a
natural sense under the modular element
allowing the passage between the left and right Haar weights.

Classically it is well known that if $G$ is a locally compact group
and $H$ is a compact subgroup of $G$, then $G$ acts continuously (by
multiplication) on the algebra of continuous functions on $G$ constant
on the cosets of $H$. Using the results obtained in the first part of
the paper we show a broad quantum counterpart of this fact -- a
coamenable locally compact quantum group coacts continuously not only
on the fixed point algebras for the canonical actions of its compact
quantum subgroups (which has been earlier proved in \cite{Soltan}),
but also on arbitrary right invariant expected $C^*$-subalgebras.

As we are interested in quantum counterparts of all regular Borel
measures on a locally compact group $G$ (and not
only those which are absolutely continuous with respect to the Haar
measures), it is natural to work with the
$C^*$-algebraic versions of locally compact quantum groups. The
standing assumption of coamenability allows us to avoid certain
technical subtleties related to the notion of a quantum subgroup and
can be explained by the desire to model the classical fact that there is a
natural one-to-one correspondence between
closed subgroups of a locally compact group $G$ and measures on $G$
induced by Haar measures on these subgroups --
in general, Haar weights on non-coamenable locally compact quantum
groups need not be faithful.

The detailed plan of the paper is as follows: in Section 1 we briefly
recall basic facts on weights, conditional
expectations for $C^*$-algebras, multiplier algebras and strict
topology, establish the notation and terminology
related to locally compact quantum groups and prove a few technical
results required later in the paper. Section~2 is devoted to
idempotent states on a locally compact quantum group
-- in particular we show that the idempotent
property of a state can be characterised by the fact that its
associated convolution operator is a conditional
expectation, and we introduce a natural partial order on the set of
idempotent states. Section 3 contains the main
results of the paper. We first establish the natural
correspondence between idempotent states and right
invariant expected $C^*$-subalgebras under the assumption of
unimodularity, then prove that Haar idempotents are
in a natural sense invariant under the modular element and finally
combine these facts to obtain the natural
correspondence between Haar idempotents and right invariant symmetric
expected $C^*$-subalgebras in the general
case. The last short section is devoted to showing that the coproduct
of a coamenable locally compact quantum
group restricts to a continuous coaction on each right invariant
expected $C^*$-subalgebra.

\section{Preliminary facts}

In this section we describe basic facts related to weight-preserving
conditional expectations in the
$C^*$-algebraic context and recall the basic definitions and facts of
the theory of locally compact quantum groups. The symbol $\ot$ denotes the minimal/spatial tensor product of $C^*$-algebras.

\subsection*{Weights and conditional expectations on $C^*$-algebras}
We begin by gathering a few facts which are standard for normal
weights on von Neumann algebras and are probably well known also in
this, $C^*$-algebraic, context.

\begin{deft}
A \emph{weight} on a $C^*$-algebra $\alg$ is a map $\psi:\alg_+ \to
[0, \infty]$ that is additive and homogenous
with respect to scalars in $\br_+$. A weight $\psi$ is \emph{densely defined} if
\[
\Mpsi^+:= \set{a\in \alg_+}{\psi (a) < \infty}
\]
is dense in $\alg_+$. It is \emph{faithful} if the equality
$\psi(a)=0$, with $a\in \alg_+$, implies that $a=0$.
\end{deft}

A weight $\psi$ on $\alg$ can be uniquely extended (as a linear
functional) to the linear span of $\Mpsi^+$; the latter is equal to
$\Npsi^* \Npsi$, where
\[
\Npsi = \set{a \in\alg}{a^* a \in \Mpsi^+}.
\]
This can be found for example in \cite{Takesaki2}.

\begin{deft}
Let $\psi$ be a weight on a $C^*$-algebra $\alg$,
and let $\clg$ be a $C^*$-subalgebra of $\alg$.
A norm-one projection $E_{\clg}$ from $\alg$ onto
$\clg$ is called a \emph{conditional expectation}. We say that $E_{\clg}$
\emph{preserves} $\psi$ if for all $a \in \Mpsi^+$
\[
E_{\clg} (a) \in \Mpsi^+, \quad \psi(E_{\clg} (a)) = \psi(a).
\]
If a conditional expectation onto $\clg$ preserving $\psi$ exists,
$\clg$ is called \emph{$\psi$-expected}.
\end{deft}

Note that if $E_{\clg}$ preserves $\psi$, then  by linearity the
equality $\psi(E_{\clg} (a)) = \psi(a)$ is valid for all $a \in \Npsi^* \Npsi$.

\begin{lem} \label{condexp}
Let $\psi$ be a densely defined faithful weight on a $C^*$-algebra
$\alg$. Suppose that $\clg$ is a $C^*$-subalgebra
of $\alg$ such that there exists a $\psi$-preserving conditional
expectation $E_{\clg}$ from $\alg$ onto $\clg$.
Then $E_{\clg}$ is uniquely determined, $E_{\clg} (\Npsi)\sub
\Npsi$ and $\Npsi \cap \clg$ is dense in $\clg$.
Moreover, if we define
$\clg^{\perp}:=\set{a \in \alg \cap \Npsi}{E_{\clg} (a) =0}$, then
\[
\clg^{\perp}
= \set{a\in\alg\cap\Npsi}{\psi(c^*a)= 0 \textup{ for all } c\in\clg\cap\Npsi}.
\]
\end{lem}

\begin{proof}
Let $a \in \Npsi$. As a conditional expectation,  $E_{\clg}$ is a
contractive completely positive map and by the Kadison--Schwarz
inequality $E_{\clg}(a^*)E_{\clg}(a) \leq E_{\clg}(a^*a)$.
Since $a^*a\in \Mpsi^+$ and $E_{\clg}$ is $\psi$-preserving, we have
\[
\psi( E_{\clg}(a^*)E_{\clg}(a)) \leq \psi(E_{\clg}(a^*a))= \psi(a^*a) < \infty.
\]
Thus $E_{\clg}(\Npsi) = \clg \cap \Npsi$ and $\clg \cap \Npsi$ is
dense in $\clg$. Further if $a \in \Npsi$ and $c \in \clg \cap \Npsi$, then
\[
\psi(c^*a) = \psi(E_{\clg}(c^*a)) = \psi (c^* E_{\clg}(a)).
\]
Hence if $E_{\clg}(a)=0$, then $\psi(c^*a)=0$.
On the other hand if $\psi(c^*a)=0$ for all $c \in \clg \cap \Npsi$,
then $0 = \psi(E_{\clg}(a)^* a) = \psi(E_{\clg}(E_{\clg}(a)^*a))= \psi(E_{\clg}(a)^*E_{\clg}(a))$ and
faithfulness of $\psi$ implies that $E_{\clg}(a)=0$.

The above arguments imply that $\Npsi = \clg^{\perp} \oplus (\clg \cap \Npsi)$ (as a vector space). As $\Npsi$ is
dense in $\alg$ and $\clg^{\perp}$ depends only on $\clg$ and $\psi$, the uniqueness of the $\psi$-preserving
conditional expectation onto $\clg$ follows.
\end{proof}

\subsection*{Multiplier algebra}
The \emph{multiplier algebra} $M(\alg)$ of a $C^*$-algebra $\alg$ is the
largest unital $C^*$-algebra that contains
$\alg$ as an essential ideal; in other words $M(\alg)$ is the largest reasonable
unitisation of $\alg$. We will write $1_{\alg}$ to denote the unit of $M(\alg)$.
The \emph{strict topology} on $M(\alg)$ is the
topology generated by the seminorms
$x\mapsto \|xa\|$, $x\mapsto \|ax\|$, where $a$
runs through the elements of $\alg$. A $^*$-homo\-mor\-phism $\pi:
\alg\to M(\blg)$, where $\blg$ is another
$C^*$-algebra, is \emph{nondegenerate} if $\pi(\alg)\blg$ is dense in
$\blg$. A $^*$-homo\-mor\-phism $\pi$ is
nondegenerate if and only if the net $(\pi(e_i))_{i \in I}$ converges strictly to
$1_\blg$ for any (or for every)
bounded approximate identity $(e_i)_{i \in I}$ of $\alg$.
Moreover, we say that a $C^*$-subalgebra $\clg$ of $\alg$
is nondegerate if the embedding $\clg\hookrightarrow \alg$
is nondegenerate. A bounded linear map $\alg\to M(\blg)$ that admits a
(unique) extension to $M(\alg)$ which is strictly continuous on bounded
sets is said to be \emph{strict} (this unique extension is denoted by
the same symbol as the original map, unless stated otherwise). Every
nondegenerate $^*$-homo\-mor\-phism is strict. So are all bounded
linear functionals $\omega: \alg\to \complex$ as well as slice maps of
bounded linear functionals: for example
$\omega\ot\id_\blg: \alg\otimes \blg\to \blg$. Also nondegenerate
completely positive maps are strict: a completely positive map
$P: \alg\to M(\blg)$ is nondegenerate if for some bounded
approximate identity $(e_i)_{i \in I}$ of $\alg$ the net $(P(e_i))_{i \in I}$ converges
strictly to $1_{\blg}$ in $M(\blg)$. The last fact
implies that if $\blg$ is a nondegenerate $C^*$-subalgebra of $\alg$
and $E$ is a conditional expectation from $\alg$ to $\blg$, then $E$
is strict. A good reference to nondegenerate completely positive maps,
as well as multiplier algebras in general, is \cite{lance}.

\begin{lem} \label{multmod}
Let $\alg$ be a $C^*$-algebra, let $\clg$ be a nondegenerate
$C^*$-subalgebra of $\alg$ and let $E:\alg \to \clg$ be a conditional
expectation onto $\clg$.
Then the strict extension $\tilde{E}$ of $E$ to $M(\alg)$
is a conditional expectation onto $M(\clg)$. In particular,
\[
\tilde{E}(m n) = \tilde{E}(m) n\qquad\text{and}\qquad
\tilde{E}(n m) = n \tilde{E}(m)
\]
for every $m \in M(\alg)$ and $n \in M(\clg)$.
\end{lem}

\begin{proof}
As  $\clg$ is assumed to be nondegenerate in $\alg$, it follows from
Proposition 2.3 of \cite{lance} that we can identify $M(\clg)$ in a canonical
way with a subalgebra of $M(\alg)$. The well-known module
property of conditional expectations states that for all $a \in \alg$,
$c \in \clg$
\[
E(ac) = E(a) c.
\]
Fix $m\in M(\alg)$. By a version of the Kaplansky theorem for strictly dense
algebras (Proposition 1.4 of \cite{lance}),
there is a bounded net $(a_i)_{i \in I}$ in $\alg$
that converges strictly to $m$.
As $\tilde{E}$ is  strictly continuous on bounded sets, we obtain
\[
E(mc) = \lim_{i \in I} E(a_i c) = \lim_{i \in I} E(a_i) c= \tilde{E}(m) c.
\]
Let then $n \in M(\clg)\sub M(\alg)$ and let $(c_i)_{i \in I}$ be a
bounded net in $\clg$ strictly convergent to $n$. By
nondegeneracy it is also strictly convergent to $n$ in
$M(\alg)$. Hence $(mc_i)_{i \in I}$ is a bounded net in $\alg$
strictly convergent to $mn$, and
\[
\tilde{E}(mn) = \lim_{i \in I} E(m c_i) =
\lim_{i \in I} \tilde{E}(m) c_i = \tilde{E}(m) n.
\]
In particular,
\[
\tilde{E}(n) = \tilde{E}(1_{\clg}n) = \tilde{E}(1_{\clg}) n = n,
\]
so $\tilde{E}$ maps $M(\alg)$ onto $M(\clg)$. Since
$\tilde{E}$ is also a projection of norm one,
it is a conditional expectation.
\end{proof}

\subsection*{Elements affiliated with a $C^*$-algebra}

The noncommutative analogue of unbounded continuous functions
is given by elements affiliated with a $C^*$-algebra.
These were introduced by Woronowicz in \cite{Woraff}
(and also by Baaj in \cite{Baajthese}).
Let $\alg$ be a $C^*$-algebra. A densely defined operator
$T: \alg\bus\dom(T)\to \alg$ is \emph{affiliated} with $\alg$
if there exists $z_T\in M(\alg)$ with $\|z_T\|\leq 1$ such that
\[
x\in \dom(T)\text{ and }y = Tx
\]
if and only if
\[
x = (1_\alg-z_T^*z_T)^{1/2}a \text{ for some }a\in \alg
\text{ and }y = z_T a.
\]
The element $z_T$, which is unique when it exists,
is called the \emph{$z$-transform} of $T$.
Elements of the multiplier algebra $M(\alg)$
are exactly the bounded operators on $\alg$
that are affiliated with $\alg$. If $\alg$ is
unital, every element affiliated with $\alg$ is in $\alg$. If
$\alg\approx C_0(X)$ for a locally compact space $X$, then
the elements affiliated with $\alg$ correspond
precisely to the continuous, not necessarily bounded, functions on~$X$.

Suppose that $\alg$ and $\blg$ are $C^*$-algebras,
$\pi: \alg\to M(\blg)$ is a nondegenerate $^*$-homo\-mor\-phism,
and $T$ is affiliated with $\alg$.
By Theorem 1.2 of \cite{Woraff}, there is a unique operator $\pi(T)$
such that $\pi(T)$  is affiliated with $\blg$,
the set $\pi(\dom(T))\blg$ is a core of $\pi(T)$ and
\begin{equation} \label{eq:homo-aff}
\pi(T)\bigl(\pi(a)b\bigr) = \pi(Ta)b
\end{equation}
for every $a$ in $\dom(T)$ and $b$ in $\blg$.
Moreover, the composition rule
\[
\rho(\pi(T)) = (\rho\circ\pi)(T)
\]
is true when both $\pi$ and $\rho$ are nondegenerate $^*$-homo\-mor\-phisms.

Woronowicz \cite{Woraff} developed functional calculus
of normal affiliated elements; see also \cite{WorNap} and
\cite{kus:funct-calc}. We will need it to
deal with powers of strictly positive elements.
An operator $T$ affiliated with $\alg$ is
\emph{strictly positive} if it is positive (i.e.\ $z_T$ is
positive) and has a dense range. Such an operator is  necessarily injective.
If $T$ is a normal element affiliated with $\alg$ and $f$ is a continuous
complex-valued function defined on the spectrum of $T$,
then $f(T)$ is a well-defined operator that is affiliated with
$\alg$. Moreover, if $\pi: \alg\to M(\blg)$ is a
nondegenerate $^*$-homo\-mor\-phism, then
\begin{equation} \label{eq:fc-comm}
f\bigl(\pi(T)\bigr) = \pi\bigl(f(T)\bigr).
\end{equation}
This is equation (1.4) of \cite{WorNap}. However, we
need to apply the above equation also to negative
powers of strictly positive operators in which case
we need to exclude $0$ from the spectrum of $T$. The equation
still holds true in that case, as shown in
Proposition 6.17 of \cite{kus:funct-calc}.
By Proposition 7.11 of \cite{kus:funct-calc}, all real powers of a
strictly positive operator affiliated with $\alg$ are
also strictly positive and affiliated with $\alg$.

\subsection*{Locally compact quantum groups in the $C^*$-algebraic
  framework}

We follow the $C^*$-algebraic approach to locally compact quantum groups
due to Kustermans and Vaes \cite{KV}.
\begin{deft}
A \emph{coproduct} on a $C^*$-algebra $\alg$ is
a nondegenerate $^*$-homo\-mor\-phism
$\Com: \alg\to M(\alg\otimes\alg)$ that
is coassociative:
\[
(\id_\alg\otimes \Com)\Com = (\Com\otimes\id_\alg)\Com.
\]
A $C^*$-algebra $\alg$ with a coproduct $\Com$
is a \emph{locally compact quantum group} if
the quantum cancellation laws
\[
\clsp\Com(\alg)(\alg\otimes 1_\alg) = \alg\ot \alg,\quad
\clsp\Com(\alg)(1_\alg\otimes\alg)  = \alg\ot \alg
\]
are satisfied and if there exist faithful KMS-weights
$\phi$ and $\psi$ on $\alg$ such that
$\phi$ is \emph{left invariant}
\[
\phi((\omega\ot\id)\Com(a)) = \omega(1_\alg)\phi(a)
\qquad (\omega\in \alg^*_+, a\in\Mphi^+)
\]
and $\psi$ is \emph{right invariant}
\[
\psi((\id\ot\omega)\Com(a)) = \omega(1_\alg)\psi(a)
\qquad (\omega\in \alg^*_+, a\in\Mpsi^+).
\]
(The KMS-condition is a technical condition which we
shall not use explicitly. For our purposes it is enough
to know that the weights are densely defined and
lower semicontinuous.)
The left invariant weight $\phi$ is called the
\emph{left Haar weight} and the right invariant $\psi$ the
\emph{right Haar weight}.
The quantum group $\alg$ is said to be \emph{unimodular} if
$\phi=\psi$. The faithfulness assumption means that we are dealing with
the reduced (as opposed to the universal, see \cite{kus:univ}) version of
$\alg$.

If the $C^*$-algebra $\alg$ is unital, we say that the quantum group
$\alg$ is \emph{compact}. In this case there is a unique \emph{Haar state}
of $\alg$ that is both left and right invariant.
Compact quantum groups are often defined so that the Haar state
is not necessarily faithful. This makes no difference in
our setting, where we consider compact quantum subgroups
of a coamenable locally compact quantum group:
see the remark at the end of this section.

A \emph{counit} is a linear functional $\epsilon$ on $\alg$ such that
\[
(\epsilon\ot\id_\alg)\Com = (\id_\alg\ot\epsilon)\Com = \id_\alg.
\]
A quantum group that has a bounded counit is said to be
\emph{coamenable} \cite{bedos-tuset}.
\end{deft}

Suppose that $\alg$ is a locally compact quantum group.
Then the GNS rep\-re\-sen\-ta\-tion associated with the left Haar weight
$\phi$ is faithful, and so we may consider $\alg$
as a $C^*$-subalgebra of $B(H)$, where $H$ is the Hilbert
space obtained from the GNS construction.

The \emph{right multiplicative unitary} $V$ is
a very special unitary operator on $H\otimes H$ which belongs to the
multiplier algebra $M(K(H)\ot \alg))$
(where $K(H)$ denotes the algebra of compact operators on $H$)
and which completely determines $\alg$. The coproduct $\Com$ is given by
\[
\Com(a) = V(a\otimes 1)V^*\qquad(a\in\alg)
\]
and the $C^*$-algebra $\alg$ itself is the norm closure of
\[
\set{(\sigma\ot\id)V}{\sigma\in B(H)_*}.
\]
Here $B(H)_*$ denotes the predual of $B(H)$: the
weak*-continuous functionals on $B(H)$.
The multiplicative unitary $V$ satisfies the pentagonal relation
\[
V_{12}V_{13}V_{23} = V_{23}V_{12}.
\]
In the above formula we use the leg numbering notation:
$V_{12} = V\otimes 1$, $V_{23} = 1\ot V$
and $V_{13} = (1\ot \Sigma)(V\ot 1)(1\ot\Sigma)$,
$\Sigma(\xi\ot\eta) = \eta\ot\xi$.

The \emph{antipode} $S$ of the quantum group $\alg$
is a closed, densely defined operator on $\alg$.
The elements $(\sigma\ot\id)V$, with $\sigma\in B(H)_*$, form a core of $S$ and
\begin{equation} \label{eq:antipode}
S((\sigma\ot\id)V) = (\sigma\ot\id)V^*.
\end{equation}
It is more common to use the formula
\[
S((\id\ot\sigma)W) = (\id\ot\sigma)W^*,
\]
where $W$ is the left multiplicative unitary; see Proposition~8.3 of \cite{KV}.
Equation \eqref{eq:antipode} follows from Proposition~8.3 of \cite{KV}
by using the identity
\[
V = (\hat J\ot\hat J)\Sigma W^* \Sigma (\hat J\ot\hat J)
\]
where $\Sigma: H\ot H \to H\ot H$ is the flip map and $\hat J: H\to H$
is the anti-unitary modular conjugation coming from the Tomita-Takesaki theory for the left Haar weight of the dual locally compact quantum group $\hat{\alg}$ (see \cite{Vaesthesis}).

When $\alg$ is not unimodular, the modular element $\delta$
provides a passage between the Haar weights $\phi$ and $\psi$.
The modular element $\delta$ is a strictly positive,
unbounded operator on $H$ that is affiliated with the
$C^*$-algebra $\alg$. Formally,
\[
\psi(a) = \phi(\delta^{1/2}a\delta^{1/2}), \;\;\; a \in \Mpsi^+ .
\]

A $C^*$-algebra $\blg$ acts in a natural way on $\blg^*$:
for each $\mu \in \blg^*, b \in\blg$ we define $\mu_b\in
\blg^*$ and $ _b \mu \in  \blg^*$ by
\[
\mu_b (c) = \mu (bc), \quad _b \mu(c) = \mu(cb).
\]

Let $\alg$ be a coamenable locally compact quantum group with
left and right invariant weights $\phi$ and $\psi$, respectively.
If $\mu, \nu \in \alg^*$ we write
\[
\mu \star \nu = (\mu \ot \nu) \Com, \qquad L_{\mu}
= (\mu \ot \id_{\alg}) \Com, \qquad R_{\mu} = (\id_{\alg} \ot \mu) \Com.
\]
The fact that the maps $L_{\mu}$ and $R_{\mu}$ take values in $\alg$
(and not just in $M(\alg)$) is a consequence of
the inclusions of both $\Delta(\alg)(\alg \ot 1_{\alg})$ and
$\Delta(\alg)(1_{\alg} \ot \alg)$ in $\alg \ot \alg$
and the usual factorisation of functionals on $\alg$ resulting from
Cohen's factorisation theorem (for all $\mu
\in \alg^*$ there exists $\nu, \nu' \in \alg^*$ and $a,a' \in \alg$
such that $\mu = \nu_a = \,_{a'} \nu'$). It
can be easily checked that if $\mu, \nu \in\alg^*$ then $L_{\mu} =
L_{\nu}$ implies $\mu = \nu$ and we always have
$L_{\mu} L_{\nu} = L_{\nu \star \mu}$.
The following lemma gives a useful characterisation for
the maps of the form $L_{\omega}$; it is closely related to
Theorem~2.4 in \cite{discrete}.

\begin{lem} \label{invar}
Let $\alg$ be a coamenable locally compact quantum group and
 let $T:\alg \to \alg$ be a completely bounded map such that $T\ot\id_{\alg}: \alg \ot \alg \to \alg \ot \alg$ is
strictly continuous on bounded subsets. Suppose that $G\sub \alg^*$ is weak$^*$-dense and invariant under the
right action of some dense subalgebra $\Alg$ of $\alg$. Then the following conditions are equivalent:
\begin{rlist}
\item $T=L_{\mu}$, for some $\mu \in \alg^*$;
\item \label{commutT}
      $(T \ot \id_{\alg} ) \Com = \Com  T$;
\item $T R_{\nu} = R_{\nu} T$ for all $\nu \in G$.
\end{rlist}
If the above conditions hold, $T= R_{\Cou \circ T}$ is a linear combination of completely positive nondegenerate
maps and the equality in (iii) is valid for all $\nu \in \alg^*$.
\end{lem}

\begin{proof}
(i) $\Rightarrow$ (ii) Standard calculation.

(ii) $\Rightarrow$ (i) Application of $\Cou \ot \id_{\alg}$ to the  equality in (ii).

(ii) $\Rightarrow$ (iii) Fix $\nu \in \alg^*$ and apply $\id_{\alg}  \ot \nu$ to  the equality in (ii).

(iii) $\Rightarrow$ (ii) Let $\nu \in G$ and $a\in \Alg$. As $\nu_a \in G$ we deduce from (iii) that for all $x \in \alg$
\[
(\id_{\alg} \ot \nu) \left((T \ot \id_{\alg})
                          \left((1_\alg \ot a)\Com(x)\right) \right)
= (\id_{\alg} \ot \nu) \left((1_\alg \ot a)\Com(T(x))\right).
\]
As the family  $\set{\mu \ot \nu}{\mu \in \alg^*, \nu \in G}$ separates points in $\alg \ot \alg$, the last
formula is equivalent to
\[
(T \ot \id_{\alg}) \left((1_\alg \ot a) \Com(x)\right) = (1_\alg \ot a)\Com(T(x))
\]
holding for all $a \in \Alg$, $x\in \alg$. That this implies (ii) is a consequence of density of $\Alg$ in $\alg$
and the definition of the multiplier algebra.
\end{proof}

A \emph{compact quantum subgroup} of $\alg$ is
a compact quantum group $\blg$ with a surjective $^*$-homo\-mor\-phism
$\pi_\blg: \alg\to \blg$ such that
$(\pi_\blg\otimes\pi_\blg)\Com_\alg = \Com_\blg\pi_\blg$,
where $\Com_\alg$ and $\Com_\blg$ denote the coproducts of
$\alg$ and $\blg$, respectively.
Although the morphism $\pi_\blg$ is an essential part of the
definition, we shall usually suppress it
and say that $\blg$ is a compact quantum subgroup of $\alg$.

Woronowicz defined compact quantum group as a unital $C^*$-algebra
with a coproduct satisfying the quantum cancellation laws and showed that such a compact quantum group
has always a unique Haar state which, however,
does not need to be faithful.
Although following the framework of \cite{KV} we defined compact quantum group more restrictively,
our setting actually covers also the more general situation.
Indeed, suppose that we have a coamenable locally compact quantum
group $\alg$ which has a Woronowicz-type compact quantum subgroup
$\blg$. Then we can quotient out the kernel of the Haar state
of $\blg$ to obtain a compact quantum subgroup $\blg'$
of $\alg$ in our sense.
By Theorem~8 of \cite{S} a compact quantum subgroup of a
coamenable locally compact quantum group is always coamenable,
and it follows that $\blg = \blg'$.

\section{Definition and basic properties of idempotent states on
locally compact quantum groups.}

From now on we assume that $\alg$ is a coamenable locally compact quantum group. The following definition is
central to the paper.

\begin{deft}
A state $\omega \in \alg^*$ is called an \emph{idempotent state} if $\omega
\star \omega = \omega$.
\end{deft}

As stated in the introduction, idempotent states on locally compact quantum groups can arise from Haar states of compact quantum subgroups. This motivates the next definition.

\begin{deft}
We call a state $\omega \in \alg^*$ a \emph{Haar idempotent} if there exists a compact quantum subgroup $\blg$ of
$\alg$ such that $\omega=h_{\blg} \circ \pi_\blg$ where $h_\blg$ is the Haar state of $\blg$ and $\pi_\blg:\alg
\to \blg$ is the surjective morphism associated with the quantum subgroup $\blg$.
\end{deft}

Each Haar idempotent is an idempotent state; in the next section we will provide a characterisation of the Haar
property in terms of various objects associated with a given idempotent state.

It is easy to see (see the beginning of the proof of Theorem \ref{idempcondex} below) that $\omega \in \alg^*$ is
an idempotent state if and only if  $L_{\omega}$ (or, equivalently, $R_{\omega}$) is a unital positive idempotent
map. It is tempting to replace the last statement by saying that $L_{\omega}$ is a conditional expectation. To
show that it is indeed possible we need some preparations.

The next lemma was proved for compact quantum groups in \cite{FS} (Lemma 3.1 there). As can be seen from its proof in the expanded
arXiv version of that paper, apart from the Cauchy--Schwarz inequality for states the only other ingredient used
is that the linear spans of $\Delta(\alg) (\alg \ot 1_{\alg})$ and $\Delta(\alg) (1_{\alg} \ot \alg)$ are dense in
$\alg \ot \alg$. Since this remains true when $\alg$ is a locally compact quantum group, and all algebraic
manipulations can be justified via strictness arguments, we obtain the following result.

\begin{lem}\label{omegac}
Let $\sigma$ and $\omega$ be states on $\alg$ and assume that $\sigma
\star \omega = \omega \star \sigma = \omega$. Let $b\in \alg$. Then
\[
\omega \star {}_b\sigma = \sigma(b) \omega.
\]
In particular, if $\omega \in \alg^*$ is an idempotent state, then
$\omega \star {}_b\omega = \omega(b) \omega$.
\end{lem}

\begin{tw} \label{idempcondex}
A functional $\omega \in \alg^*$ is an idempotent state if and only if
$L_{\omega}$ \textup{(}or, equivalently,
$R_{\omega}$\textup{)} is a conditional expectation.
\end{tw}

\begin{proof}
Assume first that  $L_{\omega}$ is a conditional expectation. Then it
is automatically completely positive, so that
$\omega=\Cou \circ L_{\omega}$ is a contractive positive functional.
As $L_{\omega} = L_{\omega} L_{\omega} = L_{\omega\star  \omega}$,
we obtain that $\omega = \omega \star \omega$. The strict extension of
this formula to $M(\alg)$ shows that
$\omega(1_{\alg}) = \omega(1_{\alg})^2$. As $L_{\omega}$, so
also $\omega$, cannot be equal to $0$, we must have
$\omega(1_{\alg}) =1$ and $\omega$ is an idempotent state.

Assume now that $\omega$ is an idempotent state. Then $L_{\omega}:\alg \to \alg$ is a completely positive nondegenerate
idempotent map. Complete positivity follows from the fact that both slice maps with positive functionals and $^*$-homomorphisms are completely positive (see for example \cite[p.4]{Was2}); nondegeneracy is a consequence of the fact that $\omega$ is a state. To show that $L_{\omega}$ is a conditional expectation it is enough to check that its image is an
algebra. To that end it suffices to show that for all $a,b \in \alg$
\begin{equation} \label{cond}
L_{\omega}(a) L_{\omega}(b) =  L_{\omega} (L_{\omega}(a)b).
\end{equation}

Let $a, b, c, d \in \alg$. Write $\Com_{2}: \alg \to M(\alg \ot \alg \ot \alg)$ for $(\Com \ot \id_{\alg})\Com =
(\id_{\alg} \ot \Com) \Com$ and compute:
\begin{align*}
(\omega \ot \omega \ot \id_{\alg})&  \big(\Com_2(a) (1_\alg \ot c \ot d)\big)
 =  (\omega \ot {}_c\omega \ot \id_{\alg}) \big(\Com_2(a) (1_\alg \ot 1_\alg \ot d)\big)\\
&= \big((\omega \star {}_c\omega) \ot \id_{\alg}\big) \big(\Com(a) (1_\alg\ot d)\big)
 = \omega(c) (\omega \ot \id_{\alg}) \big(\Com(a) (1_\alg  \ot d)\big)\\
&= L_{\omega}(a) (\omega \ot \id_{\alg})(c \ot d)
\end{align*}
(note that Lemma~\ref{omegac} was used in the third equality). A simple norm limit argument shows that if $x \in
\alg \ot \alg$, then
\begin{equation}\label{AotA}
(\omega \ot \omega \ot \id_{\alg}) \big(\Com_2(a) (1_\alg \ot x)\big)
 = L_{\omega}(a) (\omega \ot \id_{\alg})(x).
\end{equation}
Further, using the fact that $\Com(\alg) (\alg \ot 1_\alg)$ is contained in $\alg \ot \alg$ and all maps in the formula
above are strict we deduce that for each $y \in \Com(\alg)$
\[
(\omega \ot \omega \ot \id_{\alg}) \big(\Com_2(a) (1_\alg \ot y)\big)
 = L_{\omega}(a) (\omega \ot \id_{\alg})(y)
\]
(to prove it formally we apply identity \eqref{AotA} to $y(e_{i} \ot 1_\alg)$ for an approximate identity $(e_i)_{i \in I}$ and pass to
the limit). In particular,
\begin{align*}
L_{\omega}(L_{\omega}(a) b)
&= L_{\omega} \big((\omega \ot \id_{\alg}) \big(\Com(a)(1_\alg \ot b) \big) \big) \\
&= (\omega \ot L_{\omega}) \big(\Com(a)(1_\alg \ot b) \big)
 = (\omega \ot \omega \ot \id_{\alg}) \big(\Com_2(a) (1_\alg \ot \Com(b))\big)\\
&= L_{\omega}(a) (\omega \ot \id_{\alg})(\Com(b))= L_{\omega}(a) L_{\omega}(b).
\end{align*}

The arguments above can be easily adapted to work with the map $R_{\omega}$.

\end{proof}

Lemma \ref{omegac} has other useful consequences. Recall that the
\emph{multiplicative domain} of a state $\omega$ of $\alg$ is the set
\[
\set{x\in \alg}{\omega(ax)= \omega(xa) = \omega(a)\omega(x)
                \text{ for every }a\in\alg}.
\]

\begin{lem}\label{mult-dom}
Let $\omega$ be an idempotent state on $\alg$.
Then $L_\omega(\alg)$ is contained in the
multiplicative domain of $\omega$.
\end{lem}

\begin{proof}
Let $a\in\alg$. Since $L_\omega$ is a conditional expectation,
\[
\omega\bigl(L_\omega(a)^* L_\omega(a)\bigr)
= \omega\bigl(L_\omega(L_\omega(a^*)a)\bigr)
=\omega(L_\omega(a^*)a) = \omega\star {}_a\omega(a^*).
\]
By Lemma~\ref{omegac},
\[
\omega \star {}_a\omega(a^*) = \omega(a)\omega(a^*)
=\omega\bigl(L_\omega(a)^*\bigr)\omega\bigl(L_\omega(a)\bigr).
\]
Similarly,
\[
\omega\bigl(L_\omega(a) L_\omega(a)^*\bigr)
=\omega\bigl(L_\omega(a)\bigr)\omega\bigl(L_\omega(a)^*\bigr),
\]
so Choi's theorem \cite[Theorem~3.18]{paulsen}, applied to the strict extension of $\omega$ to $\tilde{\omega}\in M(\alg)^*$ implies that
$L_\omega(a)$ is in the multiplicative domain of $\tilde{\omega}$, so also in the multiplicative domain of $\omega$.
\end{proof}

In Proposition 3.4 of \cite{FST} it is shown that idempotent states on
compact quantum groups are in a sense
invariant under the antipode. The next proposition shows that the same
remains true in the locally compact context. The proof is different
from that in \cite{FST}.

\begin{propn} \label{w=wS}
Every idempotent state $\omega$ satisfies $\omega S(a) = \omega(a)$
for every $a$ in $\dom (S)$.
\end{propn}

\begin{proof}
By Theorem~\ref{idempcondex}, $L_\omega$ is a conditional
expectation onto $\clg:= L_\omega(\alg)$.
It follows that $\id\ot L_\omega: K(H)\ot \alg \to K(H)\ot\clg$
is also a conditional expectation. Next we check that
$\clg$ is nondegenerate in $\alg$: we have
\[
\clg \alg = L_{\omega} (\alg)\alg =
(\omega \ot \id_{\alg}) \big(\Com(\alg) (1_\alg \ot \alg) \big),
\]
so the quantum cancellation law $\clsp \Com(\alg) (1_\alg \ot \alg) = \alg\ot\alg$
implies that the linear span of $\clg\alg$ is dense in $\alg$.
Then Cohen's factorisation theorem implies that
$\clg\alg$ is dense in $\alg$. Since $\clg$ is nondegenerate in $\alg$,
it follows that $K(H)\ot\clg$ is nondegenerate in $K(H)\ot\alg$.
Hence the conditional expectation $\id\otimes L_\omega$
extends to a conditional expectation on $M(K(H)\ot\alg)$
by Lemma~\ref{multmod}.

Put $p = (\id\ot\omega)V$. (It turns out that $p$ is an orthogonal
projection, hence the notation.) Now
\[
(\id\ot L_\omega)V = (\id\ot\omega\ot\id)V_{12}V_{13}
=(p\ot 1)V.
\]
On the other hand, $p^* = (\id\ot\omega)V^*$ and
\[
(\id\ot L_\omega)V^* = V^*(p^*\ot 1).
\]

Then
\begin{align*}
p p^*\ot 1_{\alg} &= (p\ot 1_{\alg})VV^*(p^*\ot 1_{\alg})
= \bigl((\id\ot L_\omega)V\bigr)\bigl((\id\ot L_\omega)V^* \bigr)\\
&= (\id\ot L_\omega)\bigl(((\id\ot L_\omega)V )V^*\bigr)
\end{align*}
because $\id\ot L_\omega$ is a conditional expectation on
$M(K(H)\ot \alg)$. Continue the calculation:
\[
p p^*\ot 1 = (\id\ot L_\omega)\bigl((p\ot 1)VV^*\bigr)
=(\id\ot L_\omega)\bigl(p\ot 1 \bigr)
= p\ot 1.
\]
Hence $pp^*=p$ and $p$ is an orthogonal projection.

Now for every $\sigma$ in $B(H)_*$,
\[
\omega((\sigma\ot\id)V) = \sigma(p)
=\sigma(p^*) = \omega ((\sigma\ot\id)V^*) = \omega S((\sigma\ot\id)V).
\]
Since $\set{(\sigma\ot\id)V}{\sigma\in B(H)_*}$ is a core of $S$,
we have $\omega(a) = \omega S(a)$ for every $a$ in $\dom (S)$.
\end{proof}

Classically the subgroups of a given group form a lattice. The
resulting partial order motivated
a natural partial order $\prec$ on idempotent states on a finite
quantum group introduced and studied in \cite{FSJAlg}. Below we show
that an analogous partial ordering can be considered also for
idempotent states on a locally compact quantum group.

\begin{deft}
If $\omega, \omega'\in\alg^*$ are idempotent states, we write
$\omega\prec\omega'$ if $\omega \star \omega' =\omega'$
\end{deft}

To prove that $\prec$ is antisymmetric, we need to exploit the properties
of idempotent states established above.

\begin{propn} \label{antisymm}
Suppose that $\omega$ and $\omega'$ are idempotent states such that
$\omega\prec\omega'$, i.e. $\omega \star\omega' = \omega'$.
Then also $\omega'\star \omega = \omega'$.
Moreover, if we have also $\omega'\prec \omega$, then $\omega = \omega'$.
\end{propn}

\begin{proof}
By Proposition~\ref{w=wS}, the functionals $\omega S$ and $\omega' S$ are bounded and their bounded extensions to
all of $A$ are $\omega$ and $\omega'$, respectively. Lemma~5.25 of \cite{KV} implies that for every $a$ in $\dom
(S)$
\[
\omega'\star\omega(a) = (\omega \otimes \omega')\Com(S(a))
= \omega' S(a) = \omega'(a).
\]

As for the second statement, if $\omega\prec\omega'$, then
$\omega = \omega'\star \omega$, which by the beginning of the proof is
equal to $\omega'$.
\end{proof}

\section{Correspondence between idempotent states and right invariant
  expected $C^*$-subalgebras}

This section contains the main results of the paper, establishing the
correspondence between idempotent states and
right invariant expected $C^*$-subalgebras (under the assumption of
unimodularity) and between Haar idempotents
and right invariant symmetric expected $C^*$-subalgebras.  They  extend
on one hand Theorem 4.1 of \cite{FS} and on
the other Theorem 13 of \cite{S}. Again $\alg$ denotes a coamenable
locally compact quantum group.

\begin{deft}
A $C^*$-subalgebra $\clg \sub \alg$ is said to be \emph{right
invariant} if $R_{\mu}(\clg) \sub \clg$ for
all $\mu \in \alg^*$. It is said to be \emph{expected} if there exists
a conditional expectation $E_{\clg}$ onto
$\clg$ that is both $\psi$-preserving and $\phi$-preserving.
\end{deft}

A nondegenerate $C^*$-subalgebra $\clg$ of a coamenable
quantum group $\alg$ is right invariant if and only if
$\Com(\clg)\sub M(\clg\otimes \alg)$.
Indeed, if $\Com(\clg)\sub M(\clg\otimes \alg)$,
then $R_\mu(\clg)\sub M(\clg)\cap \alg$
for every $\mu\in\alg^*$, and the nondegeneracy of $\clg$ implies
that $R_\mu(\clg)\sub \clg$.
The converse is shown in Theorem~\ref{Soltanimproved} below,
but we can also use the so-called slice map property
to give a short argument.
Suppose that $R_\mu(c)\in\clg$ for every $\mu\in\alg^*$ and $c \in \clg$.
Then for every $a\in \alg$, $d\in\clg$ and $\nu\in\alg^*$,
\begin{equation} \label{eq:slice}
(d\ot a)\Com(c)\in \alg\ot \alg\quad\text{and}\quad
(\id_\alg\ot\nu)(d\ot a)\Com(c)\in \clg\ot \alg.
\end{equation}
Since $\alg$ is coamenable, it follows that $\alg$ is
a nuclear $C^*$-algebra \cite{bedos-tuset} and so has the slice map property
introduced by Wassermann~\cite{wassermann}.
By the slice map property, \eqref{eq:slice} implies that
$(d\ot a)\Com(c) \in \clg\ot\alg$.
The other side can be dealt with similarly, so
we see that $\Com(c) \in M(\clg\ot\alg)$.

The relation between right invariant subalgebras and idempotent states
can be described by the two propositions below.

\begin{propn} \label{id-sub}
If\/ $\omega \in \alg^*$ is an idempotent state, then $L_{\omega}(\alg)$
is a right invariant $\phi$-expected $C^*$-subalgebra of $\alg$.
\end{propn}

\begin{proof}
The statement is a direct consequence of Theorem \ref{idempcondex} and
the definition of the left Haar weight $\phi$.
\end{proof}

\begin{propn} \label{sub-id}
If $\clg$ is a right invariant $\psi$-expected $C^*$-subalgebra of
$\alg$, then there exists a unique idempotent
state $\omega\in \alg^*$ such that $\clg = L_{\omega}(\alg)$.
\end{propn}

\begin{proof}
Let $E_{\clg}$ denote the $\psi$-preserving conditional expectation
onto a right invariant $C^*$-subalgebra $\clg$. We want to show that
$E_\clg$ is of the form $L_{\omega}(\alg)$ for an idempotent state
$\omega\in \alg^*$. Due to Lemma \ref{invar} and Theorem
\ref{idempcondex} it suffices to show that
\[
E_{\clg} R_{\nu} = R_{\nu} E_{\clg}
\]
for every $\nu \in G$, where $G$ is a weak$^*$-dense subspace of $\alg^*$
that is invariant under the action of a dense
subalgebra of $\alg$.
Let $L^1(\alg)$ be the collection of $\nu\in\alg^*$
that are restrictions of weak$^*$-continuous functionals on
$B(H)$.
Moreover, let $L^1_\sharp(\alg)$ be the collection of
all $\nu\in L^1(\alg)$ for which there exists $\nu^\sharp\in L^1(\alg)$
such that $\conj\nu(S(a)) = \nu^\sharp(a)$
for every $a\in \dom (S)$ (here $\conj{\nu}$ is defined by
$\conj{\nu}(a) = \conj{\nu(a^*)}$\,). Now we consider
$G = L^1_\sharp(\alg)$ which is weak$^*$-dense in $\alg^*$ and invariant under
the right action of the dense subalgebra $\dom (S)$.
The first fact is well known \cite{KV}. To check the latter statement,
note that for $\nu\in L^1_\sharp(\alg)$, $a\in \dom (S)$ and
$x\in \dom(S)$, we have
\[
\conj{\nu_a}(S(x)) = \conj{\nu(aS(x)^*)} = \conj{\nu}(S(x)a^*)
=\conj{\nu}\bigl(S(x)S(S(a)^*)\bigr) = \nu^\sharp(S(a)^* x).
\]
Since $L^1(\alg)$ is invariant under the right action of $\alg$,
$(\nu^\sharp)_{S(a)^*}\in L^1(\alg)$, and so it follows that
$\nu_a\in L^1_\sharp(\alg)$.

Let $\nu \in L^1_\sharp(\alg)$. We need to show that
\[
E_{\clg} R_{\nu} = R_{\nu} E_{\clg}.
\]
As the maps above are linear and continuous, it suffices to show that
\begin{equation} \label{commutN}
E_{\clg} (R_{\nu}(a)) = R_{\nu} (E_{\clg}(a)),\qquad a \in \Npsi.
\end{equation}
If $a\in \clg$, then the formula above is an immediate consequence of right invariance of $\clg$. Let then $a \in
\clg^{\perp}$ (see the notation in Lemma~\ref{condexp}). Due to Lemma~\ref{condexp} it suffices to prove that
$R_{\nu}(a) \in \clg^{\perp}$. Note first that $R_{\nu}(a) \in \Npsi$ because we can decompose $\nu$ into states
and then use the Kadison--Schwarz inequality and right invariance of $\psi$. Put $\rho = \overline{\nu^{\sharp}}$
so that $\rho \circ S= \nu$. Let $c \in \clg \cap \Npsi$ and compute:
\begin{align*}
\psi (c^* R_{\nu} (a))
&= \psi \bigl( (\id_{\alg} \ot \nu)((c^* \ot 1_\alg) \Com (a)) \bigr)
 =\nu \bigl( (\psi \ot \id_{\alg})((c^* \ot 1_\alg) \Com (a)) \bigr) \\
&= \rho \circ S \bigl( (\psi \ot \id_{\alg})( (c^* \ot 1_\alg)\Com (a) ) \bigr)
 = \rho\bigl( (\psi \ot \id_{\alg})( \Com (c^*)(a \ot 1_\alg) )\bigr).
\end{align*}
The last equality holds because $\psi$ is strongly right invariant
\cite[Proposition 5.24]{KV}. Continuing the calculation, we have
\[
\psi (c^* R_{\nu} (a))
 = \psi\bigl((\id_{\alg} \ot \rho)( \Com (c^*)(a \ot 1_\alg))\bigr)
 = \psi(R_{\overline{\rho}}(c)^* a)=0,
\]
as $\clg$ is right invariant and $a\in \clg^{\perp}$. So
$R_\nu(a)\in\clg^\perp$ as required.
\end{proof}

A combination of Propositions \ref{id-sub} and \ref{sub-id} immediately yields the next corollary.

\begin{cor}
If $\clg$ is a right invariant $\psi$-expected $C^*$-subalgebra of
$\alg$ then it is expected \textup{(}the $\psi$-preserving conditional expectation onto $\clg$ preserves also $\phi$\textup{)}.
\end{cor}

The difference between the conditions in the propositions above lies
of course in the fact that in the first of
them we see the left Haar weight, and in the second the right Haar
weight. If $\alg$ is unimodular, the
distinction disappears and we obtain the following result. Recall that
the set of right invariant expected
$C^*$-subalgebras of $\alg$ is equipped with a partial order given by
the set inclusion.

\begin{tw} \label{thm:unimodular-corr}
Suppose that $\alg$ is a unimodular coamenable locally compact quantum
group. There is an order-preserving one-to-one correspondence between
idempotent states $\omega$ on $\alg$ and right invariant, expected
$C^*$-subalgebras $\clg$ of $\alg$. The correspondence is given by
\[
\clg_\omega = L_\omega(\alg),\qquad \omega_\clg = \Cou\circ E_\clg,
\]
where $E_\clg$ denotes the conditional expectation onto $\clg$ and
$\Cou$ the counit of $\alg$.
\end{tw}

\begin{proof}
The existence of the correspondence is a consequence of Propositions
\ref{id-sub} and \ref{sub-id}.  The fact that it preserves respective
partial orders is easy to check.
\end{proof}

We will show later (in Theorem \ref{Haarcorresp}) that even in the
absence of unimodularity there is a one-to-one
correspondence between Haar idempotents on a locally compact quantum
group and right invariant expected $C^*$-subalgebras of a particular
type. Before we do that we provide in the next theorem a criteria which
characterises Haar idempotents among all idempotent states. We need
another definition.

\begin{deft} A right invariant $C^*$-subalgebra $\clg$ of $\alg$ is
  said to be \emph{symmetric} if
\begin{equation} \label{conorm}
V^*(1\otimes c) V \in M(K(H)\ot \clg)
\end{equation}
for every $c$ in $\clg$.
\end{deft}

The above definition was used in \cite{S}
and a similar condition appeared already in a slightly different guise
in \cite{Reiji} under the name of `coaction symmetry':
it can be interpreted as the invariance of $\clg$ under the natural
action of the dual locally compact quantum group $\hat{\alg}$ on $B(H)$.
If $\alg$ is the group $C^*$-algebra of an amenable locally compact
group $G$ and $H$ is an open subgroup of $G$, then $C^*(H)$ is
symmetric in $C^*(G)$ if and only if $H$ is a normal subgroup
(see Section 7 in \cite{S};
related facts can be also found in Theorem 3.5 of \cite{FST}).

\begin{tw}   \label{thm:Haar-equiv}
Let $\omega$ be an idempotent state on $\alg$ and let $\clg_{\omega} =
L_\omega(\alg)$.
Then the following are equivalent:
\begin{rlist}
\item $\omega$ is a Haar idempotent; \item $\clg_{\omega}$ is symmetric;
\item $N_{\omega}:=\set{a \in \alg}{\omega(a^*a)=0}$ is an ideal
      \textup{(}equivalently a $^*$-subspace\textup{)}.
\end{rlist}
\end{tw}
\begin{proof}
(i)$\Longrightarrow$(ii) Suppose that $\blg$ is
a compact quantum subgroup of $\alg$, that $\pi_{\blg}$ is a
corresponding morphism and that $\omega=h_{\blg} \circ \pi_{\blg}$,
where $h_{\blg}$ is the Haar state of $\blg$. Then  Theorem 10 of
\cite{S} states that $\clg_\omega = L_{\omega}(\alg)$ is symmetric.

(ii)$\Longrightarrow $(i)
The left invariant $C^*$-subalgebra $\clg_\omega$ gives
rise to a compact quantum subgroup $\blg$ of
$\alg$ by Theorem~2 of \cite{S}. Let $\pi_\blg:\alg \to \blg$ be the
associated  morphism and let $h_\blg$ be the
Haar state of $\blg$. Now the Haar idempotent $\gamma:=h_{\blg}\circ
\pi_{\blg}$ gives rise to another right invariant $C^*$-subalgebra
$\clg_\gamma=L_\gamma(\alg)$. But since $\clg_{\omega}$ admits a
conditional expectation $L_{\omega}$ satisfying the invariance
condition $(L_{\omega}\ot\id_\alg)\Com = \Com L_{\omega}$, Theorem 11
of \cite{S} implies that $\clg_{\omega}=\clg_{\gamma}$. Now both
$L_\omega$ and $L_\gamma$ are $\phi$-preserving conditional
expectations onto the same $C^*$-subalgebra, so by the
uniqueness of conditional expectations preserving a fixed faithful
weight, $L_{\omega} = L_{\gamma}$. The last equality yields
$\gamma=\omega$, so that $\omega$ is a Haar idempotent.

(i)$\Longrightarrow $(iii)
Let $\omega$ be a Haar idempotent, so that there exists a compact
quantum subgroup $\blg$ of $\alg$ such that $\omega=h_{\blg} \circ \pi_{\blg}$.
Since the Haar state $h_\blg$ is faithful, we have
\begin{equation}  \label{string}
\begin{split}
a \in N_{\omega} &\iff \omega(a^*a)= 0
 \iff h_\blg(\pi_{\blg}(a)^* \pi_{\blg}(a)) = 0
 \iff \pi_{\blg}(a) =  0\\
& \iff \pi_{\blg}(a^*) = 0
 \iff \omega(aa^*)=0 \iff a^* \in N_{\omega}.
\end{split}
\end{equation}

(iii)$\Longrightarrow $(i)
Suppose that $N_{\omega}$ is an ideal. Let $\blg$ denote the
$C^*$-algebra $\alg /{N_{\omega}}$, and let $\pi_\omega:\alg \to \blg$
be the quotient map.
We claim that $\blg$ is unital. Let $e\in \alg_+$ such that $\omega(e)
= 1$. To show that $\pi_\omega(L_\omega(e))$
is a unit for $\blg$, we need to show that $a - aL_\omega(e)$ and $a -
L_\omega(e)a$ are in $N_\omega$ for every
$a$ in $\alg$. Since $L_\omega(e)$ is in the multiplicative domain of
$\omega$ by Lemma~\ref{mult-dom},
\begin{align*}
\omega\bigl((a - aL_\omega(e))^*(a - a L_\omega(e)&)\bigr)
\\&= \omega\bigl(a^* a - a^* a L_\omega(e) -L_\omega(e^*)a^* a +
         L_\omega(e^*)a^*aL_\omega(e)\bigr)
         \\
&= \omega(a^*a)\bigl(1 - \omega(L_\omega(e)) - \omega(L_\omega(e)) +
              \omega(L_\omega(e))^2\bigr)
\\
&= \omega(a^*a)\bigl(1 - 2\omega(e) +\omega(e)^2\bigr) = 0.
\end{align*}
Since $N_\omega$ is self-adjoint, also
$a-L_\omega(e)a = (a^*-a^*L_\omega(e))^*$ is in $N_\omega$.

The proof continues from now on as in the compact case \cite{FST},
with some necessary modifications. Recall
that a positive map between $C^*$-algebras is said to be faithful if
it maps positive non-zero elements into
non-zero elements. Let $(e_i)_{i \in I}$ denote an approximate unit in
$\alg$. A standard use of the Cauchy--Schwarz
inequality implies that if $a \in N_{\omega}$, then $\omega(e_i a)=0$
for all $i \in I$, and so $\omega(a)=0$. Let
$\mu \in\blg^*$ be a linear functional such that $\mu \circ \pi_\omega
= \omega$. It is obviously positive and
\[
\mu(1_\blg) = \mu\bigl(\pi_\omega(L_\omega(e))\bigr)
=\omega(e) = 1,
\]
so $\mu$ is a state.

Since $\mu\circ\pi_\omega = \omega$ and $\ker\pi_\omega = N_\omega$ it
follows that  $\mu$ is faithful. Let $x$ in $\blg\ot\blg$ be
positive. If $(\id_\blg\ot \mu)(x)=0$, then
$\mu\bigl((\nu \ot \id_\blg)(x)\bigr)=0$ for any state $\nu \in
\blg^*$, and so $(\nu \ot \id_\blg)(x)=0$. This in turn implies that
$x=0$ (replace $\nu$ by an arbitrary continuous functional,
apply another one of these and use the fact that tensor products of
functionals in $\blg^*$ separate points in
$\blg \ot \blg$). Thus we have shown that $\id_\blg \ot \mu:\blg \ot
\blg \to \blg$ is faithful. The faithfulness  of
$\mu \otimes \mu\in (\blg \ot\blg)^*$ is now an easy consequence.

Suppose now that $a \in N_{\omega}$. We have then
\[
0 = \omega(a^*a) = (\omega \otimes \omega)\circ\Com(a^*a)
= (\mu \otimes \mu)\circ (\pi_\omega \otimes \pi_\omega)(\Com(a^*a)),
\]
so also  $(\pi_\omega \otimes \pi_\omega)(\Com(a^*a))= 0$ (note that
we use here a strict extension of $\pi_\omega \otimes \pi_\omega$,
which also takes values in $\blg \ot \blg$). The last statement
implies that $(\pi_\omega \otimes \pi_\omega)\Com(a) = 0$.

The computation in the last paragraph implies that we can define a map
$\Com_\blg: \blg \to \blg \ot \blg$ via the formula
\begin{equation} \label{copb} \Com_{\blg}
\circ \pi_\omega  = (\pi_\omega \otimes \pi_\omega) \circ \Com.
\end{equation}
The fact that $\Com_{\blg}$ is a nondegenerate (hence unital)
$^*$-homo\-mor\-phism follows immediately from the
analogous properties of $\Com$. Once we know that $\Com_\blg$ is
nondegenerate we can check that it is
coassociative (all maps involved in the calculations are strict so
there are no problems in passing to multiplier
algebras).  Finally the quantum cancellation properties of $\blg$
follow from obvious equalities of the type
\[
(\blg\otimes 1_{\blg}) \Com_{\blg}(\blg)
 = (\pi_\omega\ot \pi_\omega)((\alg\ot 1_{\alg})\Com(\alg))
\]
and the fact that the quantum cancellation properties hold for
$\alg$. Thus $(\blg, \Com_{\blg})$ is a compact
quantum group. It is easy to check that $\mu$ defined above is an
idempotent state on $\blg$. As it is faithful,
it coincides with the Haar state of $\blg$ by \cite{woronowicz98}.
\end{proof}

The following corollary is an immediate consequence of the equivalence
of (i) and (iii). In the commutative case it gives the Kawada-It\^o Theorem.

\begin{cor}
If an idempotent state on a coamenable locally compact quantum group
$\alg$ is tracial, then it is a Haar
idempotent. In particular if $\alg$ is commutative, then all
idempotent states on $\alg$ are Haar idempotents.
\end{cor}

The equivalence of conditions (i)$\Longleftrightarrow$(ii) is
essentially contained in \cite{S}; in the case of
von Neumann algebraic compact quantum groups it was shown in
\cite{Reiji}.

Below we identify the ideal $N_\omega$
appearing in Theorem~\ref{thm:Haar-equiv} with an ideal
constructed in \cite{S}, but
we need some terminology first.
Let $\clg$ be a right invariant C*-subalgebra of $\alg$.
We say that a nondegenerate representation $\rho$
of $\alg$ on a Hilbert space $H_\rho$ is \emph{$\clg$-trivial} if
\[
\rho(c) = \epsilon(c)1_{H_\rho}\qquad\text{for every $c$ in }\clg.
\]
Define $J_{\clg} = \bigcap_\rho \ker\rho$
where the intersection is taken over the equivalence classes of
nondegenerate $\clg$-trivial representations of $\alg$.

\begin{propn}
Let $\omega$ be a Haar idempotent and let $\blg$ be a compact quantum
subgroup of $\alg$ such that $\omega=h_{\blg} \circ \pi_{\blg}$.
Then $N_{\omega}=J_{\clg_\omega}$ where $\clg_\omega = L_\omega(\alg)$.
\end{propn}

\begin{proof}
By the construction in Section 5 of \cite{S},
the right invariant $C^*$-subalgebra associated with the
compact quantum subgroup $\blg$ is precisely $\clg_{\omega}$.
Then the ideal $J_{\clg_\omega}$,
again by construction, is the kernel of the map
$\pi_{\blg'}: \alg \to \blg'$, where $\blg'$
is the compact quantum subgroup associated with
$\clg_{\omega}$ \cite[Theorem~2]{S}.
Now the uniqueness result \cite[Theorem~12]{S} says
that $\blg$ is isomorphic to $\blg'$. In fact, as is shown
in the proof of  \cite[Theorem~12]{S}, $\ker \pi_\blg =
\ker\pi_{\blg'}$.
But $J_{\clg_\omega} = \ker\pi_{\blg'}$ and  $N_{\omega}= \ker\pi_{\blg'}$
by \eqref{string}, so it follows that $N_{\omega} = J_{\clg_\omega}$.
\end{proof}

The remainder of this section will be devoted to the extension of
Theorem~\ref{thm:unimodular-corr} to
non-unimodular locally compact quantum groups in the case of Haar
idempotents. The proof of Theorem \ref{thm:unimodular-corr} suggests
that the only missing ingredient needed for such an extension is the
fact that the conditional expectation $L_{\omega}$ preserves not only
the left Haar weight $\phi$, but also the right Haar
weight $\psi$. It is natural to expect that to show this we need to
exploit the modular element facilitating the passage between $\phi$
and $\psi$.

The proof will be split into a series of lemmas. The first one is probably
well known and does not involve any notions related to quantum groups.

\begin{lem} \label{GNSfaith}
Let $\dlg$ be a $C^*$-algebra, let $\rho \in \dlg^*$ be a state  and
let $(\pi, H_{\rho}, \Omega)$ be the GNS construction for $\rho$.
Denote by $\rho' \in (\pi(\dlg))^*$ the vector state associated with
$\Omega$ \textup{(}so that $\rho =\rho'\circ \pi$\textup{)}.
Then $\rho'$ is faithful if and only if the null space
$N_{\rho}:=\set{d \in \dlg}{\rho(d^*d)=0}$ is an ideal.
\end{lem}

\begin{proof}
Let $d \in \dlg$. On one hand we have the following string of equivalences:
\begin{align*}
\pi(d)=0 \;& \Longleftrightarrow \;
           \forall_{c \in \dlg}\; \pi(d)\pi(c)\Omega = 0
   \; \Longleftrightarrow \;\forall_{c \in \dlg}\;
    \langle \Omega, \pi(c)^* \pi(d)^* \pi(d) \pi(c) \Omega \rangle = 0 \\
   \; &\Longleftrightarrow \; \forall_{c  \in \dlg}\;\rho(c^*d^*dc) = 0
\; \Longleftrightarrow \;\forall_{c \in \dlg}\; dc \in N_{\rho}.
\end{align*}
On the other hand $\rho'(\pi(d)^*\pi(d) ) = \rho(d^*d) = 0$
if and only if $d \in N_{\rho}$.
\end{proof}

The next lemma will be used to show that Haar idempotents are invariant
under the suitably understood action of the modular element.
Classically this corresponds to the fact that
the modular function of a locally compact group $G$
is the constant function $1$ when restricted to any compact subgroup of $G$
(because it is a continuous homomorphism into the non-negative reals).

\begin{lem} \label{affiliateddensity}
Let $\omega \in \alg^*$ be a Haar idempotent and let
$(\pi,H_{\omega}, \Omega)$ be the GNS construction for $\omega$. Then
the  $C^*$-algebra $\pi(\alg)$ is unital. If\/ $t$ is a
positive self-adjoint operator affiliated with $\alg$ such that
$\Com(t)= t \ot t$, then the positive functional  $\omega_t \in
\alg^*$ defined by the formula
\begin{equation} \label{omegat}
\omega_t(a) = \langle \Omega, \pi(t) \pi(a) \pi(t) \Omega \rangle,
\qquad a \in \alg,
\end{equation}
is either zero or an idempotent state.
Moreover, if $t$ is strictly positive,
then actually $\pi(t) = 1$ and $\omega_t = \omega$.
\end{lem}

\begin{proof}
Similarly to the previous lemma let $\omega'\in \pi(\alg)^*$ be such that
$\omega = \omega' \circ \pi$.  Choose $e\in \alg_+$ such that $\omega(e)=1$.
We claim that $\pi (L_{\omega}(e))$ is a unit for $\pi(\alg)$. Indeed,
by the implication (i)$\Longrightarrow$(iii) of Theorem  \ref{thm:Haar-equiv} and the previous lemma it suffices to show that both
$\pi (L_{\omega}(e) a)- \pi(a)$ and $\pi (a L_{\omega}(e) ) - \pi(a)$
belong to $\set{b \in \pi(\alg)}{\omega'(b^*b)=0}$
for arbitrary $a \in \alg$. This however can be established exactly as
in the proof of the implication (iii)$\Longrightarrow$(i) in
Theorem~\ref{thm:Haar-equiv}.

Since $\pi$ is a nondegenerate $^*$-homo\-mor\-phism and
$t$ is affiliated with $\alg$, it follows
that $\pi(t)$ is affiliated with $\pi(\alg)$.
But $\pi(\alg)$ is unital, so $\pi(t)\in \alg$.
In particular $\pi(t)$ is bounded, and so the formula \eqref{omegat}
defines a bounded functional on $\alg$.

Let $D_t = \set{a \in \alg}{at \textrm{ is bounded}}$.
Let us note that since $t$ is positive, $D_t = (\dom(t))^*$.
Indeed, suppose first that $at$ is bounded.
For every $b$ in $\dom(t)$ we have $a(tb) = (at)b$.
Therefore, by definition, $a^*\in\dom(t^*)$ and $t^*a^* = (at)^*$.
Since $t = t^*$, we have $ta^* = (at)^*$.
Conversely, suppose that $a^*\in\dom(t)$.  Then for every $b$
in $\dom(t)$  we have $a(tb) = (t a^*)^*b$, so $at = (ta^*)^*$
on $\dom(t)$. Hence $at$ is bounded.

For $a\in D_t$ we have
\[
\begin{split}
(\omega_t \star \omega_t) (a^*a)
&= \langle \Omega \ot \Omega, (\pi \ot \pi) (t \ot t) (\pi \ot \pi)
  (\Com(a^*a)) (\pi \ot \pi) (t \ot t) (\Omega \ot \Omega) \rangle\\
&= \pair{\Omega\ot\Omega, (\pi \ot \pi)(\Com(t))
     (\pi\ot\pi)(\Com(a^*a))(\pi\ot\pi) (\Com(t)) (\Omega \ot \Omega)}.
\end{split}
\]
By \eqref{eq:homo-aff},
\[
(\pi \ot \pi)(\Com(t))\bigl((\pi\ot\pi)(\Com(a^*))u\bigr)
= (\pi \ot \pi)(\Com(ta^*))u
\]
for every $u\in\pi(\alg)\ot\pi(\alg)$. On both sides of the equation
we have bounded operators evaluated at $u$ so
\[
(\pi \ot \pi)(\Com(t))(\pi\ot\pi)(\Com(a^*)) = (\pi\ot\pi)(\Com(ta^*))
\]
in $\pi(\alg)\ot\pi(\alg)$. It also follows that
\[
(\pi\ot\pi)(\Com(a))(\pi \ot \pi)(\Com(t)) = (\pi\ot\pi)(\Com(at)).
\]
Inserting these into the expansion of
$(\omega_t \star \omega_t)(a^*a)$, we have
\begin{align*}
(\omega_t \star \omega_t) (a^*a)
&= \pair{\Omega \ot \Omega, (\pi\ot\pi)(\Com(ta^*at))(\Omega\ot\Omega)}
= (\omega \ot \omega) (\Com(ta^*at)) \\
&= \omega (ta^*at) = \pair{\Omega, \pi(ta^*at)\Omega} = \omega_t (a^*a)
\end{align*}
As $\set{a^*a}{a\in D_t}$ is dense in $\alg_+$, by linearity and continuity we
deduce that $\omega_t \star \omega_t = \omega_t$. Hence $\omega_t$ is
either $0$ or a state.

Suppose now that $t$ is strictly positive. Then $t^{-1}$ is a
strictly positive operator affiliated with $\alg$.
Hence  $\pi(t^{-1})\in\pi(\alg)$ so it follows that $\pi(t)$ is invertible
with an inverse in $\pi(\alg)$. Let then $a\in \alg$ be such that
$\pi(a) = \pi(t^{-1})$. Then $\omega_t(a^2) = \omega'(1) = 1$.
So $\omega_t$ is nonzero and hence a state. Therefore
$\omega'(\pi(t)^2) =  \omega_t(1) = 1$.

By functional calculus, $t^2$ is a strictly positive element
affiliated with $\alg$. Moreover
\[
\Com(t^2) = \Com(t)^2 = (t\ot t)^2 = t^2\ot t^2
\]
by \eqref{eq:fc-comm} and \cite[Proposition~13.16]{kus:funct-calc}.
It follows that the preceding argument can be applied to
$t^2$ instead of $t$, and so $\omega'(\pi(t)^4)=1$. Hence
\[
\omega'\bigl( (\pi(t)^2 - 1)^* (\pi(t)^2 -1) \bigr) = 0,
\]
and by the faithfulness of $\omega'$ (Lemma \ref{GNSfaith}) it follows
that $\pi(t)^2=1$. Positivity and uniqueness of square roots imply
that $\pi(t)=1$ and $\omega_t =\omega$.
\end{proof}

\begin{propn} \label{rightHaarpres}
Let $\omega \in \alg^*$ be a Haar idempotent. The conditional
expectation $L_{\omega}$ preserves the right Haar weight.
\end{propn}

\begin{proof}
Denote the modular element affiliated with $\alg$ by $\delta$. To
simplify the notation we write $x:=\delta^{\frac{1}{2}}$. So $x$ is
defined by functional calculus of affiliated elements
and is itself affiliated with $\alg$. We need to show that the
completely positive operator $L_{\omega}:\alg \to \alg$ preserves
the right invariant weight $\psi$. Our approach is based on the fact
that formally
\[
\psi = \phi(x\cdot x).
\]

Following Kustermans \cite[Section~8]{kus:kms} and Vaes \cite[p.\ 325]{Vaesthesis}, put
\[
D_x^{\phi} = \set{a\in \alg}{ax\text{ is bounded and }\overline{ax}\in \Nphi}.
\]
(The overline denotes closure.) Then $D_x^{\phi}$ is a dense left ideal in $\alg$
\cite[Result 8.6]{kus:kms}. Let $a\in D_x^{\phi}$. It follows from
\cite[Corollary 8.35]{kus:kms} that $a\in \Npsi$ (because
$\psi = \phi(x\cdot x)$). Moreover, $xa^* = (ax)^*\in \Nphi^*$
(in particular, $xa^*$ is defined everywhere and is in $\alg$).

We shall show that any $a$ in $D_x^{\phi}$ satisfies
\begin{equation} \label{eq:L-inv}
\overline{x L_\omega(a^*a)x} = L_{\omega}(\overline{ax}^*\overline{ax})
\end{equation}
In particular, $L_\omega(a^*a)\in \Mpsi^+$ because $\overline{ax}^*\overline{ax}\in \Mphi$ (by \cite[Corollary
8.35]{kus:kms} again). Then we may calculate
\begin{align*}
\psi(a^*a) = \phi(\overline{xa^*ax}) = \phi(\overline{ax}^*\overline{ax}) =
\phi(L_\omega(\overline{ax}^*\overline{ax}))
= \phi(\overline{x L_\omega(a^*a)x}) = \psi(L_\omega(a^*a)).
\end{align*}
Note that the elements of the form $a^*a$, $a \in D_{x}^{\phi}$, are dense in $\alg_+$, and so it follows that $\gamma:=\psi\circ
L_\omega$ is a nonzero, densely defined, lower semicontinuous weight on $\alg$. We show next that $\gamma$ is
right invariant. For any $a \in \mathfrak{M}_{\gamma}$ and a state $\eta \in \alg_+^*$, we have
\[
\gamma \bigl((\id_\alg \ot \eta)(\Com(a))\bigr) = \psi (L_{\omega}R_{\eta}(a)) =
\psi(R_{\eta}L_{\omega}(a))
\]
Now we use the fact that for invariant weights the invariance holds in
a strong sense: $\psi (R_{\eta}(x)) = \psi(x)$ for all $x \in \alg_+$.
So we continue:
\[
\gamma \bigl((\id_\alg \ot \eta)(\Com(a))\bigr)= \psi (L_{\omega}(a))= \gamma (a).
\]
By the uniqueness of right Haar weight \cite[Theorem~7.15]{KV}, $\gamma$ is a
scalar multiple of $\psi$. The scalar has to be
$1$ because $\gamma$ and $\psi$ coincide on a dense subset. Therefore
$L_\omega$ preserves the right Haar weight.

We proceed to prove the identity \eqref{eq:L-inv}.
Since the modular element $\delta$ is strictly positive and
$\Com(\delta) = \delta\ot\delta$ \cite[Proposition~7.12]{KV},
it follows from functional calculus and \cite[Proposition~13.16]{kus:funct-calc} that
\[
\Com(\delta^z) = \delta^z\ot\delta^z
\]
for any $z$ in $\complex$.
Let $(\pi, H_{\omega}, \Omega)$ be the GNS construction
for $\omega$ and let $\omega'\in \pi(\alg)^*$ the vector
state associated with $\Omega$.
Then Lemma~\ref{affiliateddensity} implies that
\[
\omega(a) = \omega'(\pi(x)\inv\pi(a)\pi(x)\inv)\qquad(a\in \alg),
\]
To make the presentation more transparent, we
denote below the functional defined by the
right-hand side of the identity above as $\widetilde{\omega}$.

Let $a\in D_x^{\phi}$. As noted, $xa^*ax$ is bounded and densely defined: its
extension is $(xa^*)\overline{ax} = (xa^*)(xa^*)^*$. Now
\begin{equation}\label{eq:start}
\begin{split}
&L_{\widetilde \omega}(\overline{xa^*ax})
=L_{\widetilde \omega}(xa^*\overline{ax})\\
&\quad= (\omega'\otimes \id_\alg)\bigl((\pi(x)\inv\ot 1_\alg)
     (\pi\ot\id_\alg)(\cop(xa^*\overline{ax})) (\pi(x)\inv\ot 1_\alg)\bigr)\\
&\quad= (\omega'\otimes \id_\alg)\bigl((\pi(x)\inv\ot 1_\alg)
  (\pi\ot\id_\alg)(\cop(xa^*))(\pi\ot\id_\alg)(\cop(\overline{ax}))
 (\pi(x)\inv\ot 1_\alg)\bigr)
\end{split}
\end{equation}

Now both $\cop$ and $\pi\ot\id_\alg$ are nondegenerate $^*$-homo\-mor\-phisms
so $(\pi\ot\id_\alg)\cop(x)$ is affiliated with $\pi(\alg)\ot \alg$ and
\[
(\pi\otimes\id_\alg)(\cop(xa^*))u =
(\pi\otimes\id_\alg)(\cop(x))((\pi\otimes \id_\alg)\cop(a^*)u)
\]
for every $u$ in $\pi(\alg)\otimes \alg$ (see \eqref{eq:homo-aff}).
Since the linear span of $(\pi(\alg)\otimes \alg)(H_\omega\ot H)$ is dense in
$H_\omega\ot H$, it follows that
\[
(\pi\otimes\id_\alg)(\cop(xa^*)) =
(\pi\otimes\id_\alg)(\cop(x))(\pi\otimes \id_\alg)(\cop(a^*))
\]
in $B(H_\omega\ot H)$.

Moreover,
\[
(\pi\otimes\id_\alg)(\cop(x)) = \pi(x)\otimes x
\]
because $\cop(x)= x\otimes x$ (apply Theorem~6.1 of \cite{WorNap} and
Theorem 1.2 of \cite{Woraff}). So we have
\[
(\pi\otimes\id_\alg)(\cop(xa^*)) =
(\pi(x)\otimes x)((\pi\otimes \id_\alg)\cop(a^*)).
\]

Finally,
\begin{equation}\label{eq:xa*}
(\pi(x)\inv\ot 1_\alg)(\pi\otimes\id_\alg)(\cop(xa^*))
= (1_\alg\otimes x)((\pi\otimes \id_\alg)\cop(a^*)).
\end{equation}

On the other hand, since $\overline{ax} = (xa^*)^*$,
\[
(\pi\ot\id_\alg)(\cop(\overline{ax})) (\pi(x)\inv\ot 1_\alg) =
\bigl( (\pi(x)\inv\ot 1_\alg)(\pi\otimes\id_\alg)(\cop(xa^*)) \bigr)^*,
\]
which is, by \eqref{eq:xa*}, the closure of the  bounded densely
defined operator
\[
(\pi\otimes \id_\alg)(\cop(a))(1_\alg\otimes x).
\]

Continuing from \eqref{eq:start}, we have
\begin{align*}
L_{\widetilde \omega}(\overline{xa^*ax})
&=(\omega'\ot\id_\alg)\bigl((1_\alg\ot x)((\pi\ot\id_\alg)(\cop(a^*))
                  \overline{(\pi\ot \id_\alg)(\cop(a))(1_\alg\ot x)}
                              \bigr)\\
&= \overline{x(\omega'\ot\id_\alg)\bigl((\pi\ot\id_\alg)(\cop(a^*))
                            (\pi\ot \id_\alg)(\cop(a))\bigr)x}\\
&= \overline{x(\omega'\ot \id_\alg)\bigl((\pi\ot\id_\alg)(\cop(a^*a))\bigr)x}\\
&= \overline{x L_\omega(a^*a)x}.
\end{align*}
Therefore \eqref{eq:L-inv} holds.
\end{proof}

\begin{tw} \label{Haarcorresp}
Let $\alg$ be a coamenable locally compact quantum group.
There is an order-preserving one-to-one correspondence between Haar
idempotents $\omega$ on $\alg$ and symmetric,
right invariant, expected $C^*$-subalgebras $\clg$ of $\alg$. The
correspondence is given by
\[
\clg_\omega = L_\omega(\alg),\qquad \omega_\clg = \Cou\circ E_\clg,
\]
where $E_\clg$ denotes the conditional expectation onto $\clg$ and
$\Cou$ the counit of $\alg$.
\end{tw}

\begin{proof}
The theorem is a consequence of Propositions \ref{id-sub},
\ref{sub-id}, \ref{rightHaarpres} and Theorem \ref{thm:Haar-equiv}.
\end{proof}

Let $\alg$ be as in the previous theorem.
It is shown in  \cite{S} that there is a one-to-one correspondence
between compact quantum subgroups of $\alg$
and symmetric, right invariant $C^*$-subalgebras $\clg$ of $\alg$
that have a conditional expectation $E_\clg:\alg\to\clg$ such that
$(E_\clg\otimes \id_\alg)\circ\Com = \Com\circ E_\clg$.
(We have changed `left invariant' of  \cite{S}
to `right invariant' to conform with the current terminology.)
The previous theorem replaces the condition
$(E_\clg\otimes \id_\alg)\circ\Com = \Com\circ E_\clg$
with a formally weaker one: that $E_\clg$ is
invariant under both left and right Haar weights.
The latter condition is also much more natural.

We would like to finish the section with a few more general
remarks. Given a locally compact quantum group $\alg$,
the determination of all of its compact quantum subgroups is a difficult
technical problem; in fact the answer to this question does not seem
to be known for any of the fundamental examples of
genuine quantum groups such as the quantum $E(2)$ \cite{worE2}, quantum $az+b$
\cite{worAzb} or quantum $\widetilde{SU(1,1)}$ \cite{SU11}. Our
results suggest a possible approach to this question via first
determining all idempotent states on these quantum groups. This idea
was successfully applied in \cite{FST} in the context of compact
quantum groups, for example for $U_q(2)$.
Of course to apply the techniques developed in this paper, we need to
know that the quantum group in question is coamenable.
The result below is probably very well-known to the experts,
but as we could not find an explicit reference, we sketch a short proof.

\begin{tw}
The locally compact quantum groups quantum $E(2)$,
quantum \hbox{$az+b$} and  quantum \hbox{$ax+b$}
are coamenable.
\end{tw}

\begin{proof}
As all the quantum groups in question were originally defined using the
language of multiplicative unitaries, we first need to note that they
all fit into the setup of \cite{KV}, i.e.\ that they admit faithful
invariant KMS-weights. For the last two examples this has been
established in \cite{worHaar} and in \cite{VanDaele}; for the first in
\cite{BaajE2} and in \cite{VanDaele}.

Hence it suffices to observe that all these locally compact quantum
groups admit bounded counits. In each case the potential definition
for the counit can be guessed from the explicit formulas for the
coproduct in terms of the (unbounded) generators. So for example for
quantum $az+b$ we see that the counit, if it exists, must take value
$1$ at $a$ and $0$ at $b$ (precisely speaking we should consider here
an extension of the potential counit to the algebra of affiliated
elements, but this does not affect the argument). The fact that in
each of the three cases the deduced prescription indeed defines a
continuous character follows from the universal properties of the
underlying $C^*$-algebras with respect to specific commutation
relations (these universal properties are established respectively in
Theorem 1.1 of \cite{worE2}, Proposition 4.2 of \cite{worAzb} and
Proposition 3.2 of \cite{worZak}). We leave the details to the
reader.
\end{proof}

Note that the coamenability of quantum $E(2)$ has been established in
the PhD thesis of Jacobs \cite{Jacobs} (the terminology used there is
different), and its counit is also explicitly mentioned in
\cite{worE2}. We do not know whether the quantum $\widetilde{SU(1,1)}$ is
coamenable.

\section{Further properties of right invariant $C^*$-subalgebras}

In this short section we show that the algebras of the form
$L_{\omega}(\alg)$, where $\omega$ is an idempotent state on a
locally compact quantum group $\alg$, admit a natural coaction of
$\alg$. Such subalgebras generalise quantum homogenous spaces which
are naturally associated with compact quantum subgroups of $\alg$ (see
Proposition \ref{Lomega-homogen} or Section 5 in \cite{Soltan}).

Assume again that $\alg$ is a coamenable locally compact quantum group
and let $\omega\in \alg^*$ be a Haar idempotent corresponding to a
fixed  compact quantum subgroup of $\alg$, denoted by $\blg$. The
algebra $\clg_{\omega}=L_{\omega}(\alg)$ should be thought of as the
algebra of functions in $\alg$ which are invariant under the action of
$\blg$. This is formalised in the following proposition.

\begin{propn}\label{Lomega-homogen}
Suppose that $\blg$ is a compact quantum subgroup of $\alg$ and $\pi$
is the surjective morphism from $\alg$ to $\blg$.
Let $\omega=h_{\blg} \circ \pi$ be the Haar idempotent associated with $\blg$. Then
\begin{equation}  \label{invLw}
L_{\omega}(\alg)=\set{a \in \alg}{(\pi \ot \id_{\alg})(\Com(a)) = 1_{\blg}\ot a}.
\end{equation}
\end{propn}
\begin{proof}
Fix $a \in \alg$. If $(\pi \ot \id_{\alg}) (\Com(a)) =  1_{\blg} \ot a$, then
\[ L_{\omega}(a) = (h_{\blg} \ot \id_{\alg})\big((\pi \ot \id_{\alg}) (\Com(a))\big) =  (h_{\blg} \ot \id_{\alg}) (1_{\blg} \ot a) = a, \]
so $a \in L_{\omega}(\alg)$.

On the other hand
\begin{align*}
(\pi \ot \id_{\alg}) &(\Com(L_{\omega}(a))) = (\omega \ot \pi \ot \id_{\alg}) (\Com_2(a)) =
(h_{\blg}\, \pi \ot \pi \ot \id_{\alg}) (\Com_2(a)) \\&= (h_{\blg}\ot \id_{\blg} \ot \id_{\alg})(\Com_{\blg} \ot \id_{\alg})(\pi \ot \id_{\alg}) (\Com(a))\\&= (h_{\blg} (\cdot) 1_{\blg} \ot \id_{\alg})(\pi \ot \id_{\alg}) (\Com(a))
= 1_{\blg} \ot L_{\omega}(a),
\end{align*}
so that if $a \in L_{\omega}(\alg)$ then $(\pi \ot \id_{\alg}) (\Com(a)) =  1_{\blg} \ot a$.
\end{proof}

In Theorem 5.1 of \cite{Soltan} So\l tan studies the properties of the algebra appearing on the right side of
equality \eqref{invLw} (he actually works with its right version). In the next theorem we show that the properties
established in that theorem do not really depend on the fact that the idempotent state $\omega$ appearing on the
left side of \eqref{invLw} is a Haar idempotent and provide a simpler proof.

\begin{tw}\label{Soltanimproved}
Let $\alg$ be a coamenable locally compact quantum group and let
$\clg$ be a right invariant $\psi$-expected
$C^*$-subalgebra of $\alg$. Then $\clg$ is a nondegenerate
$C^*$-subalgebra of $\alg$. Moreover the map
$\alpha:=\Com|_{\clg}$ is a nondegenerate $^*$-homomorphism from
$\clg$ to $M(\clg\ot \alg)$ such that
\begin{equation}\label{Coucommut} (\id_{\clg} \ot \Cou)\alpha =
  \id_{\clg},\end{equation}
\begin{equation}\label{residual} \overline{\tu{span}}\, \alpha(\clg)
  (1_\clg \ot \alg) = \clg \ot \alg.
\end{equation}
\end{tw}
\begin{proof}
By Proposition \ref{sub-id} the algebra $\clg$ is of the form
$L_{\omega}(\alg)$ for some idempotent state $\omega \in \alg^*$.
Therefore, by the beginning of the proof of Lemma~\ref{w=wS},
$\clg$ is a nondegenerate $C^*$-subalgebra of $\alg$ (and $1_\clg = 1_\alg$).
Once the nondegeneracy of $\clg$ in $\alg$ has been established, to
show that $\alpha$ is a nondegenerate
$^*$-homomorphism from $\clg$ to $M(\clg\ot \alg)$ it suffices to
prove that it takes values in $M(\clg \ot \alg)$. We need to
show that for $c,d \in \clg$ and $a \in \alg$ the element
$x:=\Com(c)(d \ot a)$ is in $\clg \ot \alg$  (note that we already
know that $x$ belongs to $\alg \ot \alg$). Recall that by
coassociativity we have $(L_{\omega} \ot \id_{\alg}) \Com = \Com L_{\omega}$.
By Theorem \ref{condexp} $L_{\omega}$ is a conditional expectation
from $\alg$ onto $\clg$. This implies that $L_{\omega} \ot \id_{\alg}$
is also a conditional expectation, this time from $\alg \ot \alg$ onto
$\clg \ot \alg$.  Use the `extended' module property of conditional
expectations  (Lemma \ref{multmod}) to observe that
\[ (L_{\omega} \ot \id_{\alg}) (x)  = (L_{\omega} \ot \id_{\alg})
(\Com(c))\, (d \ot a) = \Com (L_{\omega}(c)) (d \ot a) =
\Com(c) (d \ot a), \]
so that  $(L_{\omega} \ot \id) (x) = x$ and $x \in \clg \ot \alg$.

The formula \eqref{Coucommut} is a direct consequence of the defining
property of the counit. Property \eqref{residual} can be proved along
the lines of the last paragraph -- we need to use the fact that for $a,b \in \alg$
\[
(L_{\omega} \ot \id_\alg)\big(\Com(a)(1_\alg \ot b)\big) =
(L_{\omega} \ot \id_\alg)(\Com(a))(1_\alg \ot b).
\]
Once this is noted, we simply observe that
\begin{align*}
\overline{\tu{span}}\,\alpha(\clg)  (1_\alg \ot \alg) &=
\overline{\tu{span}}\, \set{\Com(L_{\omega}(a)) (1_\alg \ot b)}{a,b\in\alg} \\
&= \overline{\tu{span}}\,
   \set{(L_{\omega}\ot\id_{\alg})\big(\Com(a)(1_\alg \ot b)\big)}{a,b \in
     \alg}\\
& = (L_{\omega}\ot \id_{\alg})(\alg \ot \alg) = \clg \ot \alg.
\end{align*}
\end{proof}

Using the (right-handed version of the) terminology introduced in
Definition 2.6 of \cite{imprim} the last theorem can be reformulated
in the following way.

\begin{cor}
If $\alg$ is a  coamenable locally compact quantum group and $\clg$ a
right invariant $\psi$-expected $C^*$-subalgebra of $\alg$, then the
coproduct of $\alg$ restricts to a continuous coaction of $\alg$ on $\clg$.
\end{cor}

Note finally that the algebra in \eqref{invLw} satisfies the
assumptions of Theorem \ref{Soltanimproved} by Propositions
\ref{id-sub} and \ref{rightHaarpres}, so Theorem \ref{Soltanimproved}
indeed generalizes Theorem 5.1 in \cite{Soltan}.

\vspace*{0.5cm}

\noindent \emph{Acknowledgments.}
The work on this paper was started during a special semester on
`Banach algebra and operator space techniques in topological group
theory' in Leeds in May/June 2010, funded by the EPSRC grant
EP/I002316/1. 
Both authors would like to thank the organisers of that semester for 
providing excellent research environment. 
We also thank Nico Spronk for support which, in particular,
enabled AS to visit University of Waterloo in December 2010. AS acknowledges
useful conversations with Piotr So\l tan regarding  the last part of
the paper. PS was partially supported by Academy of Finland.
We thank the referee for helpful comments.

\end{document}